\documentclass[12pt, reqno]{amsart}
\usepackage{amsmath, amsthm, amscd, amsfonts, amssymb, graphicx, color}
\usepackage[bookmarksnumbered, colorlinks, plainpages]{hyperref}

\newtheorem{theorem}{Theorem}[section]
\newtheorem{corollary}[theorem]{Corollary}
\newtheorem{lemma}[theorem]{Lemma}
\newtheorem{prop}[theorem]{Proposition}
\theoremstyle{definition}

\theoremstyle{remark}
\newtheorem{rem}[theorem]{Remark}

\numberwithin{equation}{section}

\begin{document}
\title{Eigenvalue extensions of Bohr's inequality }
\author[Jagjit Singh Matharu, Mohammad Sal Moslehian, Jaspal Singh Aujla]{Jagjit Singh Matharu$^1$, Mohammad Sal
Moslehian$^2$ and Jaspal Singh Aujla$^1$}
\address{$^{1}$ Department of Mathematics, National Institute of Technology,
Jalandhar 144011, Punjab, India.}
\email{matharujs@yahoo.com, aujlajs@nitj.ac.in}
\address{$^{2}$Department of Pure Mathematics, Center of Excellence in
Analysis on Algebraic Structures (CEAAS), Ferdowsi University of Mashhad, P.
O. Box 1159, Mashhad 91775, Iran.}
\email{moslehian@ferdowsi.um.ac.ir and moslehian@ams.org}

\begin{abstract}
We present a weak majorization inequality and apply it to prove eigenvalue
and unitarily invariant norm extensions of a version of the Bohr's inequality due to
Vasi\'c and Ke\v{c}ki\'c.
\end{abstract}

\maketitle

\section{Introduction and Preliminaries}

Let $\mathcal{M}_{n}$ denote the $C^{\ast}$-algebra of $n\times n$ complex
matrices and let $\mathcal{H}_{n}$ be the set of all Hermitian matrices in $%
\mathcal{M}_{n}$. We denote by $\mathcal{H}_{n}(J)$ the set of all
Hermitian matrices in $\mathcal{M}_{n}$ whose spectra are contained in an
interval $J \subseteq \mathbb{R}$. By $I_{n}$ we denote the identity matrix of $\mathcal{M}_{n}$.
For matrices $A, B \in \mathcal{H}_{n}$ the order relation $A \leq B$ means
that $\langle Ax, x\rangle\leq \langle Bx,x\rangle$ for all $x \in \mathbb{C}%
^n$. In particular, if $0 \leq A$, then $A$ is called positive semidefinite.

 For $A\in \mathcal{H}_{n}$, we shall always denote by $\lambda
_{1}(A)\geq \lambda _{2}(A)\geq \cdots \geq \lambda _{n}(A)$ the
eigenvalues of $A$ arranged in the decreasing order with their
multiplicities counted. By $s_{1}(A)\geq s_{2}(A)\geq \cdots \geq s_{n}(A)$,
we denote the eigenvalues of $|A|=(A^{\ast }A)^{1/2}$, i.e., the singular values
of $A$. A norm $\left\vert \left\vert \left\vert \cdot \right\vert
\right\vert \right\vert $ on $\mathcal{M}_{n}$ is said to be unitarily
invariant if $\left\vert \left\vert \left\vert UAV\right\vert \right\vert
\right\vert =\left\vert \left\vert \left\vert A\right\vert \right\vert
\right\vert $ for all $A\in \mathcal{M}_{n}$ and all unitary matrices $%
U,V\in \mathcal{M}_{n}$. The Ky Fan norms, defined as $\Vert A\Vert
_{(k)}=\sum_{j=1}^{k}s_{j}(A)$ for $k=1,2,\cdots ,n$, provide a significant
family of unitarily invariant norms. The Ky Fan dominance theorem states
that $\Vert A\Vert _{(k)}\leq \Vert B\Vert _{(k)}\,\,(1\leq k\leq n)$ if and
only if $\left\vert \left\vert \left\vert A\right\vert \right\vert
\right\vert \leq \left\vert \left\vert \left\vert B\right\vert \right\vert
\right\vert $ for all unitarily invariant norms $\left\vert \left\vert \left\vert \cdot \right\vert \right\vert
\right\vert $. For
more information on unitarily invariant norms the reader is referred to \cite%
{bhatiaa}.

The classical Bohr's inequality \cite{bohr} states that for any $z,w\in
\mathbb{C}$ and for $p,q>1$ with $\displaystyle\frac{1}{p}+\frac{1}{q}=1$,
\begin{equation*}
\vert z+w\vert^{2}\leq p\vert z\vert ^{2}+q\vert w\vert ^{2}
\end{equation*}%
with equality if and only if $w=(p-1)z$. Several operator generalizations of
the Bohr inequality have been obtained by some authors (see \cite{A-B-P,
chan, cheung, hirzallah, M-R, ZHA, Z-F}). In \cite{vasic}, Vasi\'c and Ke%
\v{c}ki\'c gave an interesting generalization of the inequality of the
following form
\begin{equation}
\left\vert \sum_{j=1}^{m}z_{j}\right\vert ^{r}\leq \left( \sum_{j=1}^{m}p
_{j}^{\frac{1}{1-r}}\right) ^{r-1}\sum_{j=1}^{m}p _{j}\vert z_{j}\vert ^{r},
\label{vasin}
\end{equation}
where $z_{j}\in \mathbb{C},~p _{j}>0,~r>1$. See also \cite{M-P-P} for an
operator extension of this inequality.

 In this paper, we aim to give a weak majorization inequality and
apply it to prove eigenvalue and unitarily invariant norm extensions of \eqref{vasin}.


\section{Generalization of Bohr's inequality}

\label{Bohr} In this section we shall prove a matrix analogue of the
inequality (\ref{vasin}). We begin with the definition of the positive
linear map.

A ${\ast}$-subspace of $\mathcal{M}_{n}$ containing $I_{n}$ is called an
operator system. A map $\Phi: \mathcal{S}\subseteq \mathcal{M}_{n} \to
\mathcal{T}\subseteq \mathcal{M}_{m}$ between two operator systems is called
positive if $\Phi (A)\geq 0$ whenever $A\geq 0$, and is called unital if $%
\Phi (I_{n})=I_{m}$. Let $[A_{ij}]_{k},$ $A_{ij}\in \mathcal{M}_{n},$ $1\leq
i,j\leq k,$ denote the $k\times k$ block matrix. Then each map $\Phi $ from
$\mathcal{S}$ to $\mathcal{T}$ induces a map $\Phi _{k}$ from $\mathcal{M}%
_{k}( \mathcal{S})$ to $\mathcal{M}_{m}(\mathcal{T})$ defined by $\Phi
_{k}\left( [A_{ij}]_{k}\right) =\left[ \Phi (A_{ij})\right] _{k}.$ We say
that $\Phi$ is completely positive if the maps $\Phi _{k}$ are positive for
all $k=1, 2, \cdots$.

To prove our main result we need Lemma \ref{lem2} which is of independent
interest. To achieve it, we, in turn, need some known lemmas.

\begin{lemma}
\cite[Theorem 4]{stine}\label{lem2.1} Every unital positive linear map on a
commutative $C^{\ast}$-algebra is completely positive.
\end{lemma}


\begin{lemma}
\cite[Theorem 1]{stine}\label{lem22} Let $\Phi $ be a unital completely
positive linear map from a $C^{\ast}$-subalgebra $\mathcal{A}$ of $\mathcal{M%
}_{n} $ into $\mathcal{M}_{m}$. Then there exist a Hilbert space $\mathcal{K}
$, an isometry $V: {\mathbb{C}}^m \to {\mathcal{K}}$ and a unital ${\ast}$%
-homomorphism $\pi$ from ${\mathcal{A}}$ into the $C^{\ast}$-algebra $B({%
\mathcal{K}})$ of all bounded linear operators such that $%
\Phi(A)=V^{\ast}\pi(A)V$.
\end{lemma}


\begin{lemma}
\label{convexle} Let $A\in \mathcal{H}_{n}(J)$ and let $f$ be a convex
function on $J,~0\in J$, $f(0)\leq 0$. Then for every vector $x\in \mathbb{C}%
^{n},$ with $\Vert x\Vert \leq 1$,
\begin{equation*}
f\left( \left\langle Ax,x\right\rangle \right) \leq \left\langle
f(A)x,x\right\rangle .
\end{equation*}
\end{lemma}

\begin{proof}
If $x=0$ then result is trivial. Let us assume that $x\neq 0.$ A well-known
result \cite[Theorem 1.2]{fmps} states that if $f$ is a convex function on
an interval $J$ and $A\in \mathcal{H}_{n}(J)$, then $f\left( \left\langle
Ay,y\right\rangle \right) \leq \left\langle f(A)y,y\right\rangle $ for all
unit vectors $y$. For $\Vert x\Vert \leq 1$, set $y=x/\Vert x\Vert $. Then
\begin{eqnarray*}
f\left( \langle Ax,x\rangle \right) &=&f\left( \Vert x\Vert ^{2}\langle
Ay,y\rangle +(1-\Vert x\Vert ^{2})0\right) \\
&\leq &\Vert x\Vert ^{2}f\left( \langle Ay,y\rangle \right) +(1-\Vert x\Vert
^{2})f(0)\qquad \quad (\mbox{by the convexity of $f$}) \\
&\leq &\Vert x\Vert ^{2}\langle f(A)y,y\rangle +(1-\Vert x\Vert
^{2})f(0)\qquad \qquad (\mbox{by  \cite[Theorem 1.2]{fmps}}) \\
&\leq &\left\langle f(A)x,x\right\rangle \,.\qquad \qquad \qquad \qquad
\qquad \qquad \qquad \quad \quad (\mbox{by $f(0)\leq 0$})
\end{eqnarray*}
\end{proof}


Now we are ready to prove the following result.

\begin{lemma}
\label{lem2} Let $A\in \mathcal{H}_{n}(J)$ and let $f$ be a convex function
defined on $J,~0\in J$, $f(0)\leq 0$. Then for every vector $x\in \mathbb{C}%
^{m}$ with $\|x\| \leq 1$ and every positive linear map $\Phi $ from $%
\mathcal{M}_{n}$ to $\mathcal{M}_{m}$ with $0<\Phi (I_{n})\leq I_{m}$,
\begin{equation*}
f(\langle \Phi (A)x,x\rangle )\leq \langle \Phi (f(A))x,x\rangle .
\end{equation*}
\end{lemma}

\begin{proof}
Let $\mathcal{A}$ be the unital commutative $C^{\ast}$-algebra generated by $%
A$ and $I_n$. Let $\Psi (X)=\Phi (I_{n})^{-\frac{1}{2}}\Phi (X)\Phi
(I_{n})^{-\frac{1}{2}}$, $X\in \mathcal{A}$. Then $\Psi $ is a unital
positive linear map from $\mathcal{A}$ to $\mathcal{M}_{m}$. Therefore by
Lemma \ref{lem2.1}, $\Psi $ is completely positive. It follows from Lemma %
\ref{lem22} that there exist a Hilbert space $\mathcal{K}$, an isometry $V: {%
\mathbb{C}}^m \to {\mathcal{K}}$ and a unital ${\ast}$-homomorphism $\pi: {%
\mathcal{A}}\to B({\mathcal{K}})$ such that $\Psi(A)=V^{\ast}\pi(A)V$. Since
$\pi $ is a representation, it commutes with $f$. For any vector $x\in
\mathbb{C}^{m}$ with $\|x\| \leq 1$, $\|V\Phi (I_{n})^{1/2}x\|\leq 1$. We
have
\begin{align*}
f(\langle \Phi (A)x,x\rangle )& =f(\langle \Phi (I_{n})^{1/2}\Psi (A)\Phi
(I_{n})^{1/2}x,x\rangle ) \\
& =f(\langle \Phi (I_{n})^{1/2}V^{\ast}\pi (A)V\Phi (I_{n})^{1/2}x,x\rangle )
\\
& =f(\langle \pi (A)V\Phi (I_{n})^{1/2}x,V\Phi (I_{n})^{1/2}x\rangle ) \\
& \leq \langle f(\pi (A))V\Phi (I_{n})^{1/2}x,V\Phi (I_{n})^{1/2}x\rangle
\qquad \qquad \quad (\mbox{by  Lemma \ref{convexle}}) \\
& =\langle \pi (f(A))V\Phi (I_{n})^{1/2}x,V\Phi (I_{n})^{1/2}x\rangle \\
& =\langle \Phi (I_{n})^{1/2}V^{\ast}\pi (f(A))V\Phi (I_{n})^{1/2}x,x\rangle
\\
& =\langle \Phi (f(A))x,x\rangle\,.
\end{align*}
\end{proof}


\begin{rem}
We can remove the condition $0\in J$ in Lemma \ref{lem2} and assume that $%
\|x\|=1$, if we assume that $\Phi$ is unital. To observe this, one may
follow the same argument as in the proof of Lemma \ref{lem2} and use \cite[%
Theorem 1.2]{fmps}.
\end{rem}


\begin{lemma}
\cite[Pg. 67]{bhatiaa}\label{lem12} Let $A \in \mathcal{H}_n$. Then
\begin{equation*}
\sum_{j=1}^k \lambda_j (A)= \max \sum_{j=1}^k \langle A x_j,x_j
\rangle\qquad (1\le k \le n),
\end{equation*}
where the maximum is taken over all choices of orthonormal vectors $%
x_1,x_2,\cdots,x_k.$
\end{lemma}


\begin{theorem}
\label{thm1} Let $f$ be a convex function on $J,~0\in J,$ $f(0)\leq 0$ and $%
A\in \mathcal{H}_{n}(J)$. Then
\begin{equation*}
\sum_{j=1}^{k}\lambda _{j}\left(f\left(\sum_{i=1}^\ell \alpha_i \Phi
_i(A)\right)\right)\leq \sum_{j=1}^{k}\lambda
_{j}\left(\sum_{i=1}^\ell\alpha_i \Phi_i(f(A))\right)\qquad (1\leq k\leq m)
\end{equation*}%
for positive linear maps $\Phi _{i},~i=1, 2, \cdots, \ell$ \ from $\mathcal{M%
}_{n}$ to $\mathcal{M}_{m}$ such that $0 \leq
\sum_{i=1}^\ell\alpha_i\Phi_i(I_{n})\leq I_{m}$, $\alpha_i \geq 0.$
\end{theorem}

\begin{proof}
Let $\lambda _{1},\lambda _{2},\cdots ,\lambda _{m}$ be the eigenvalues of $%
\sum_{i=1}^{\ell }\alpha _{i}\Phi _{i}(A)$ with $u_{1},u_{2},\cdots ,u_{m}$
as an orthonormal system of corresponding eigenvectors arranged such that $%
f(\lambda _{1})\geq f(\lambda _{2})\geq \cdots \geq f(\lambda _{m})$. We
have
\begin{align*}
\sum_{j=1}^{k}\lambda _{j}\left( f\left( \sum_{i=1}^{\ell }\alpha _{i}\Phi
_{i}(A)\right) \right) & =\sum_{j=1}^{k}f\left( \left\langle \left(
\sum_{i=1}^{\ell }\alpha _{i}\Phi _{i}(A)\right) u_{j},u_{j}\right\rangle
\right) \\
& \leq \sum_{j=1}^{k}\left\langle \left( \sum_{i=1}^{\ell }\alpha _{i}\Phi
_{i}(f(A))\right) u_{j},u_{j}\right\rangle \quad (\mbox{by Lemma \ref{lem2}})
\\
& \leq \sum_{j=1}^{k}\lambda _{j}\left( \sum_{i=1}^{\ell }\alpha _{i}\Phi
_{i}(f(A))\right) \qquad \quad \quad (\mbox{by Lemma \ref{lem12}})
\end{align*}%
for $1\leq k\leq m$.
\end{proof}


The following result is a generalization of \cite[Theorem 1]{km}.

\begin{corollary}
\label{cornew} Let $A_{1},\cdots ,A_{\ell }\in \mathcal{H}_n$ and $%
X_{1},\cdots ,X_{\ell }\in \mathcal{M}_n$ such that
\begin{equation*}
\sum_{i=1}^{\ell }\alpha_{i}X_{i}^{\ast }X_{i}\leq I_n,
\end{equation*}%
where $\alpha_i> 0$ and let $f$ be a convex function on $\mathbb{R}$, $%
f(0)\leq 0$ and $f(uv)\leq f(u)f(v)$ for all $u, v \in \mathbb{R}$. Then
\begin{equation}
\sum_{j=1}^{k}\lambda _{j}\left(f\left(\sum_{i=1}^{\ell }X_{i}^{\ast
}A_{i}X_{i}\right)\right) \leq \sum_{j=1}^{k}\lambda _{j}\left(
\sum_{i=1}^{\ell }\alpha_if(\alpha_i^{-1})X_{i}^{\ast }f(A_{i})X_{i}\right)
\label{in22}
\end{equation}
for $1\leq k\leq n.$
\end{corollary}

\begin{proof}
To prove inequality (\ref{in22}), if necessary, by replacing $X_{i}$ by $%
X_{i}+\epsilon I_n$, we can assume that the $X_{i}$'s are invertible.\newline
Let $A\in \mathcal{M}_{\ell n}$ be partitioned as $\left(
\begin{array}{ccc}
A_{11} & \cdots & A_{1\ell } \\
\vdots &  & \vdots \\
A_{\ell 1} & \cdots & A_{\ell \ell }%
\end{array}%
\right) ,$ $A_{ij}\in \mathcal{M}_n,$ $1\leq i,j\leq \ell ,$ as an $\ell
\times \ell $ block matrix. Consider the linear maps $\Phi _{i}:\mathcal{M}%
_{\ell n}\longrightarrow \mathcal{M}_n,i=1,\cdots ,\ell ,$ defined by $\Phi
_{i}(A)=X_{i}^{\ast }A_{ii}X_{i},~i=1,\cdots ,\ell .$ Then $\Phi _{i}$'s are
positive linear maps from $\mathcal{M}_{\ell n}$ to $\mathcal{M}_n$ such
that
\begin{equation*}
0 \leq \sum_{i=1}^{\ell }\alpha _{i}\Phi _{i}(I_{\ell n})=\sum_{i=1}^{\ell
}\alpha _{i}X_{i}^{\ast }X_{i}\leq I_n\,.
\end{equation*}%
Using Theorem \ref{thm1} for the diagonal matrix $A=\mbox{diag}%
(A_{11},\cdots ,A_{\ell \ell })$, we have
\begin{equation*}
\sum_{j=1}^{k}\lambda _{j}\left( f\left( \sum_{i=1}^{\ell }\alpha
_{i}X_{i}^{\ast }A_{ii}X_{i}\right)\right) \leq \sum_{j=1}^{k}\lambda
_{j}\left( \sum_{i=1}^{\ell }\alpha _{i}X_{i}^{\ast }f(A_{ii})X_{i}\right)
\qquad (1\leq k\leq n).
\end{equation*}%
Replacing $A_{ii}$ by $\alpha_i^{-1}A_{i}$ in the above inequality, we get
\begin{equation*}
\sum_{j=1}^{k}\lambda _{j}\left(f\left(\sum_{i=1}^{\ell }X_{i}^{\ast
}A_{i}X_{i}\right)\right) \leq \sum_{j=1}^{k}\lambda _{j}\left(
\sum_{i=1}^{\ell }\alpha_if(\alpha_i^{-1})X_{i}^{\ast }f(A_{i})X_{i}\right)
\qquad (1\leq k\leq n)\,,
\end{equation*}
since by an easy application of the functional calculus $f(\alpha_i^{-1}A_i)%
\leq f(\alpha_i^{-1})f(A_i)$.
\end{proof}


Now we obtain the following eigenvalue generalization of inequality (\ref%
{vasin}) as promised in the introduction.

\begin{theorem}
\label{cor4.5} Let $A_{1},\cdots ,A_{\ell }\in \mathcal{H}_n$ and $%
X_{1},\cdots ,X_{\ell }\in \mathcal{M}_n$ be such that
\begin{equation*}
\sum_{i=1}^{\ell }p_{i}^{1/1-r}X_{i}^{\ast }X_{i}\leq \sum_{i=1}^{\ell
}p_{i}^{1/(1-r)}I_n,
\end{equation*}%
where $p_{1},\cdots ,p_{\ell }>0,r>1$. Then
\begin{equation*}
\sum_{j=1}^{k}\lambda _{j}\left( \left\vert \sum_{i=1}^{\ell }X_{i}^{\ast
}A_{i}X_{i}\right\vert ^{r}\right) \leq \left( \sum_{i=1}^{\ell }p_{i}^{%
\frac{1}{1-r}}\right) ^{r-1}\sum_{j=1}^{k}\lambda _{j}\left(
\sum_{i=1}^{\ell }p_{i}X_{i}^{\ast }\vert A_{i}\vert ^{r}X_{i}\right)
\end{equation*}
for $1\leq k\leq n.$
\end{theorem}

\begin{proof}
Apply Corollary \ref{cornew} to the function $f(t)=\vert t\vert^r$ and $%
\alpha _{i}=\displaystyle\frac{p_{i}^{1/1-r}}{\sum_{i=1}^{\ell
}p_{i}^{1/(1-r)}}$.
\end{proof}


\begin{corollary}
\label{zh} Let $A_{1},\cdots ,A_{\ell }\in \mathcal{H}_n$. Then
\begin{equation}
\left\vert \left\vert \left\vert ~\text{\ }\left\vert \sum_{i=1}^{\ell
}A_{i}\right\vert ^{r}\text{ }~\right\vert \right\vert \right\vert \leq
\left\vert \left\vert \left\vert ~\sum_{i=1}^{\ell }p_{i}^{-1}\vert
A_{i}\vert ^{r}~\right\vert \right\vert \right\vert  \label{1}
\end{equation}%
for $1<r\leq 2,$ $0 < p_{1},\cdots ,p_{\ell }\leq 1$ with $%
\sum_{i=1}^{\ell }p_{i}=1.$
\end{corollary}

\begin{proof}
Taking $X_{i}=I_n,~1\leq i\leq \ell$ in Theorem \ref{cor4.5} and using that $%
\displaystyle\left( \sum_{i=1}^{\ell }p_{i}^{\frac{1}{r-1}}\right)
^{r-1}\leq \sum_{i=1}^{\ell }p_{i}=1,$ we have
\begin{equation}
\sum_{j=1}^{k}\lambda _{j}\left( \left\vert \sum_{i=1}^{\ell
}A_{i}\right\vert ^{r}\right) \leq \sum_{j=1}^{k}\lambda _{j}\left(
\sum_{i=1}^{\ell }p_{i}^{-1}\vert A_{i}\vert ^{r}\right)\qquad (1\leq k\leq
n).  \label{eqhere}
\end{equation}%
Now from (\ref{eqhere}) and the Ky Fan Dominance Theorem, it follows that
\begin{equation*}
\left\vert \left\vert \left\vert ~\text{\ }\left\vert \sum_{i=1}^{\ell
}A_{i}\right\vert ^{r}\text{ }~\right\vert \right\vert \right\vert \leq
\left\vert \left\vert \left\vert ~\sum_{i=1}^{\ell }p_{i}^{-1}\vert
A_{i}\vert ^{r}~\right\vert \right\vert \right\vert .\qedhere
\end{equation*}
\end{proof}


Next we show that the inequality (\ref{1}) can be improved when $A,B\in
\mathcal{M}_{n}$ in the case when $r\geq 2$.

\begin{lemma}[\protect\cite{as}]
\label{ausinew}Let $f$ be an increasing convex function on $J$. Then
\begin{equation*}
\lambda _{j}\left( f\left( \sum_{i=1}^{\ell }p_{i}A_{i}\right) \right) \leq
\lambda _{j}\left( \sum_{i=1}^{\ell }p_{i}f(A_{i})\right)\qquad (1\leq j\leq
n)
\end{equation*}%
for all $A_{1},\cdots ,A_{\ell }\in \mathcal{H}_{n}(J)$ and $0\leq
p_{1},\cdots ,p_{\ell }\leq 1$ such that $\displaystyle\sum_{i=1}^{%
\ell}p_{i}=1$.
\end{lemma}


\begin{prop}
\label{thmfulll} Let $A_{1},\cdots ,A_{\ell }\in \mathcal{M}_{n}$ and $r\geq
2$. Then
\begin{equation}  \label{new2}
\lambda _{j}\left( \left\vert \sum_{i=1}^{\ell }A_{i}\right\vert ^{r}\right)
\leq \lambda _{j}\left( \sum_{i=1}^{\ell }p_{i}^{1-r}\left\vert
A_{i}\right\vert ^{r}\right) \qquad (1\leq j\leq
n)
\end{equation}
for all $0 <  p_{1},\cdots ,p_{\ell }\leq 1$ such that $\displaystyle%
\sum_{i=1}^{\ell}p_{i}=1$.
\end{prop}

\begin{proof}
Clearly
\begin{equation}
\sum_{i,j=1}^{\ell }p_{i}p_{j}\left( A_{i}-A_{j}\right) ^{\ast
}(A_{i}-A_{j})\geq 0.  \label{ine33vasi}
\end{equation}%
It follows by a direct calculation that inequality%
\begin{equation}
\left\vert \sum_{j=1}^{\ell }p_{j}A_{j}\right\vert ^{2}\leq \sum_{j=1}^{\ell
}p_{j}|A_{j}|^{2}  \label{newv}
\end{equation}%
is equivalent to (\ref{ine33vasi}). Therefore \eqref{newv} holds. Due to the
function $f(t)=t^{r/2}$ is an increasing convex function, we have
\begin{align*}
\lambda _{j}\left( \left\vert \sum_{i=1}^{\ell }p_{i}A_{i}\right\vert
^{r}\right) & =\lambda _{j}^{r/2}\left( \left\vert \sum_{i=1}^{\ell
}p_{i}A_{i}\right\vert ^{2}\right) \\
& \leq \lambda _{j}^{r/2}\left( \sum_{i=1}^{\ell }p_{i}\left\vert
A_{i}\right\vert ^{2}\right) \\
& \quad (%
\mbox{by Weyl's monotonicity principal \cite[P. 63]{bhatiaa}
applied to \eqref{newv}}) \\
& =\lambda _{j}\left( \left( \sum_{i=1}^{\ell }p_{i}\left\vert
A_{i}\right\vert ^{2}\right) ^{r/2}\right) \\
& \leq \lambda _{j}\left( \sum_{i=1}^{\ell }p_{i}\left\vert A_{i}\right\vert
^{r}\right) \qquad \qquad \qquad \qquad \qquad \quad (%
\mbox{by Lemma
\ref{ausinew}})
\end{align*}%
for $1\leq j\leq n$. Now, we replace $A_{i}$ by $A_{i}/p_{i}$ to get %
\eqref{new2}.
\end{proof}


\begin{rem}
Corollary \ref{zh} and Proposition \ref{thmfulll} are generalizations of
\cite[Theorem 7]{ZHA}.
\end{rem}


\end{document}